\documentclass[doubleblind]{ecai}
\usepackage{graphicx}
\usepackage{latexsym}

\usepackage{amsmath}
\usepackage{amsthm}
\usepackage{booktabs}
\usepackage{algorithm}
\usepackage{algorithmic}
\usepackage[switch]{lineno}

\DeclareMathOperator*{\argmax}{argmax}
\usepackage{amssymb}

\usepackage{xcolor}

\newtheorem{theorem}{Theorem}
\newtheorem{lemma}{Lemma}
\newtheorem{observation}{Observation}

\begin{document}

\begin{frontmatter}

\title{Optimizing Chance-Constrained Submodular Problems \\with Variable Uncertainties}

\author[A]{\fnms{Xiankun}~\snm{Yan}\thanks{Corresponding Author. Email: xiankun.yan@adelaide.edu.au.}}
\author[A]{\fnms{Anh Viet}~\snm{Do}}
\author[B,C]{\fnms{Feng}~\snm{Shi}} 
\author[D]{\fnms{Xiaoyu}~\snm{Qin}} 
\author[A]{\fnms{Frank}~\snm{Neumann}}

\address[A]{Optimisation and Logistics, University of Adelaide, Adelaide, Australia}
\address[B]{School of Computer Science and Engineering, Central South University, Changsha, P.R. China}
\address[C]{Xiangjiang Laboratory, Changsha, P.R. China}
\address[D]{University of Birmingham, Birmingham, UK}

\begin{abstract}
Chance constraints are frequently used to limit the probability of constraint violations in real-world optimization problems where the constraints involve stochastic components.
We study chance-constrained submodular optimization problems, which capture a wide range of optimization problems with stochastic constraints. 
Previous studies considered submodular problems with stochastic knapsack constraints in the case where uncertainties are the same for each item that can be selected.
However, uncertainty levels are usually variable with respect to the different stochastic components in real-world scenarios, 
and rigorous analysis for this setting is missing in the context of submodular optimization.
This paper provides the first such analysis for this case, 
where the weights of items have the same expectation but different dispersion.
We present greedy algorithms that can obtain a high-quality solution, i.e.,
a constant approximation ratio
to the given optimal solution from the deterministic setting.
In the experiments, we demonstrate that the algorithms perform effectively on several chance-constrained instances of the maximum coverage problem and the influence maximization problem.
\end{abstract}

\end{frontmatter}

\begin{abstract}
Chance constraints are frequently used to limit the probability of constraint violations in real-world optimization problems where the constraints involve stochastic components.
We study chance-constrained submodular optimization problems, which capture a wide range of optimization problems with stochastic constraints. 
Previous studies considered submodular problems with stochastic knapsack constraints in the case where uncertainties are the same for each item that can be selected.
However, uncertainty levels are usually variable with respect to the different stochastic components in real-world scenarios, 
and rigorous analysis for this setting is missing in the context of submodular optimization.
This paper provides the first such analysis for this case, 
where the weights of items have the same expectation but different dispersion.
We present greedy algorithms that can obtain a high-quality solution, i.e.,
a constant approximation ratio
to the given optimal solution from the deterministic setting.
In the experiments, we demonstrate that the algorithms perform effectively on several chance-constrained instances of the maximum coverage problem and the influence maximization problem. 
\end{abstract}

\section{Introduction}
Stochastic components can significantly affect the quality of solutions for a given stochastic optimization problem.
Reducing the uncertain effect of stochastic components is vital 
to avoid potentially disruptive incidents in the complex and expensive system.  
\textit{Chance constraints} can be applied to optimization tasks, which limit the probability of incidental constraint violations~\cite{charnes1959chance,iwamura1996genetic,miller1965chance,poojari2008genetic}.  
A chance-constrained 
optimization problem can be described as finding an optimal solution subject to the condition that the constraints are only violated with a given small probability. 
Recently, the problem has been investigated widely~\cite{doerr2020optimization,neumann2020optimising,neumann2019runtime,neumann2022runtime,shi2022runtime,xie2019evolutionary,xie2020specific,xie2021runtime}.
A typical technique for taking chance constraints into account for a given optimization problem is 
to convert the stochastic constraints to their respective deterministic equivalents for a given conﬁdence level, which is possible when considering normally distributed stochastic components.

\textit{Submodular functions}~\cite{nemhauser1978analysis}
capture problems of diminishing returns which frequently appear in real-world scenarios.
They constitute a significant category of optimization challenges.
In the artificial intelligence literature, greedy algorithms\cite{das2011submodular,doerr2020optimization,friedrich2019greedy,zhang2016submodular} and Pareto optimization approaches\cite{neumann2020optimising,roostapour2022pareto,qian2017subset,qian2017subsets} based on evolutionary multi-objective algorithms have been widely examined for submodular optimization problems.
The goal for a submodular optimization problem with a given knapsack constraint is to find a set of elements with the maximal value of the submodular function whose total weight does not exceed the budget of the given knapsack.
There are many analyses on the deterministic version of this submodular optimization problem~\cite{khuller1999budgeted,krause2014submodular,nemhauser1978analysis,qian2017subset}.
Often the weights of elements might be stochastic and sampled from a probability distribution. 
The \textit{Chance-constrained Submodular Problem} ~\cite{chen2019chance} has been proposed to model this case. Here, the goal of the problem is to maximize a given monotone submodular function subject to the constraint that the probability of violating the knapsack constraint is no more than a small threshold value.
For this problem, Chen and Maehara~\cite{chen2019chance} reduced the chance constraint of the problem into multiple deterministic constraints by guessing the parameters and relaxed the knapsack budget and threshold. 
They rigorously analyzed an algorithm that enumerates all parameters 
for the abstracted problem with random weights sampled from arbitrary known distributions, which meets the relaxed constraint.
Doerr et al.,~\cite{doerr2020optimization} investigated a specific variant of the problem where the weights are sampled from a uniform distribution with an identical dispersion. 
They applied the one-sided Chebyshev’s inequality and a Chernoff bound separately to construct the surrogates that helps to estimate the probability of constraint violation.
In addition, they empirically showed that using the greedy algorithms based on such surrogates gives high-quality solutions in stochastic scenarios.
Furthermore, multi-objective evolutionary algorithms have also been employed to tackle this problem, e.g., the GSEMO algorithm~\cite{neumann2020optimising}.
It has been theoretically analyzed and found to achieve comparable performance to the greedy algorithms in the worse case within polynomial time.
However, uncertainties of items usually vary between the items 
and greedy algorithms with theoretical performance guarantees missing in the literature.
Such an analysis is supposed to be more challenging than the one carried out in \cite{doerr2020optimization}
since variable uncertainties of the weights lead to more intricate effects than identical uncertainties, which is reflected in the surrogate based on one-sided Chebyshev’s inequality.

In this paper, we focus on a general setting of the problem studied in \cite{doerr2020optimization,neumann2020optimising}, 
i.e., the weights of the elements are sampled from uniform distributions 
with the same expectation but different dispersion values, 
instead of from an identical uniform distribution.
We remark the item's dispersion as its uncertainty level,
such that the uncertainty level varies from item to item in this setting. 
The one-sided Chebyshev's inequality is used to construct a surrogate of the chance constraint.
In addition to the greedy algorithm (GA)
and the generalized greedy algorithm 
(GGA), our analysis also encompasses another studied greedy algorithm, the generalized greedy$+$Max algorithm (GGMA)~\cite{yaroslavtsev2020bring}.
Our rigorous analysis demonstrates that 
the GA struggles to effectively obtain an acceptable solution in the worse case
due to a heavy impact arising from the variable uncertainties.
For the GGA and the GGMA, we first use a simple strategy for element selection, 
which only considers the sum of the dispersion values.
The algorithms cannot guarantee a high-quality solution in some linear instances.
Instead of using this simple strategy,
simultaneously considering the expectation and the dispersion is promising to fill this gap. 
We adopt an improved strategy that applies the surrogate of the chance constraints in selecting elements.
Using this strategy, the GGA and the GGMA can obtain 
a $(1/2-o(1))(1-1/e)-$approximation and 
$(1/2-o(1))-$approximation, respectively.
Finally, we empirically analyze the performance of the algorithms 
on twelve chance-constrained instances of the maximum coverage problem 
and the influence maximization problem. 
The empirical results show that
the GGA and the GGMA beat the GA in most instances.
Furthermore, the GGMA obtains a solution of similar quality as the GGA,
which verifies and supplements our theoretical results.

The paper is structured as follows. 
Sections~\ref{sec:pl}-\ref{sec:algs} introduce the studied problem and the algorithms. 
Our theoretical results of the investigated algorithms are shown in Sections~\ref{sec:pga}-\ref{sec:pggma}. 
We present our experimental results in Section~\ref{sec:exp} and finish with some conclusions in Section~\ref{sec:con}.

\section{Preliminaries}
\label{sec:pl}
\subsection{Problem Definition}
Consider a set $V = \{1,...,n\}$, in which each element $i \in V$ has a weight $w_i$, and a function $f:V'\to \mathbb{R}_{\geq 0}$ defined on the subsets $V' \subseteq V$.
The function $f$ is {\it monotone} iff for any two subsets $S, T\subseteq V$ with $S \subseteq T$, $f(S) \le f(T)$ holds. 
Besides, the function $f$ is {\it submodular} iff for any two subsets $S$, $T\subseteq V$ with $S\subseteq T$ and any element $e \notin T$, 
\begin{equation}
\label{equ:1}
    \centering
    f(S\cup \{e\})-f(S) \ge f(T\cup \{e\}) - f(T).
\end{equation}

Given a \textit{monotone} submodular function $f$ defined on the subsets of $V$ and a budget $B$, the problem named 
{\it the submodular problem with respect to $V$ and $B$} 
is to look for a subset $S \subseteq V$ such that $f(S)$ is maximized and $W(S) \le B$, where $W(S) = \sum_{i \in S} w_i$.  

Within the investigation, we focus on a {\it chance-constrained} version of the submodular problem, in which the weight $w_i$ of each element $i \in V$ is random (not deterministic) and has expected value $E[w_i] = a_i$ and variance $\sigma^2_i\geq 0$. The aim is find a subset $S \subseteq V$ such that $f(S)$ is maximized and subject to the constraint that $Pr[W(S)>B] \le \alpha$, where the threshold $0 \le \alpha \le 1$ is given which upper bounds the probability of a constraint violation. 

As mentioned above, given an instance of the chance-constrained submodular problem, the weight $W(S) = \sum_{i\in S} w_i$ of a solution $S$ to it is random but has an expectation
$E[W(S)] = \sum_{i\in S} a_i$,
and variance
$Var[W(S)] = \sum_{i\in S}\sigma^2_i$.

Within the paper, we consider a specific setting for the chance-constrained submodular problem, 
in which the random weight $w_i$ of the element $i\in V$ is independently uniformly sampled from the interval $[a_i-\delta_i, a_i+\delta_i]$ at random. 
Besides, we consider $a_i =1$ and $0\leq\delta_i\leq 1$. 
Note that, for a uniform distribution, the expectation and variance can be calculated by the given interval bounds~\cite{xie2019evolutionary}.
Therefore the expected weights of all elements are $1$, and the variance of each element $i$ is $Var[W(i)] = \delta_i^2/3$.
Furthermore, w.l.o.g., assume every single element is feasible with respect to the budget $B$. 
Observe that the case that $0 \le B \le 1$ is meaningless, thus we assume $B>1$ such that at least one item is in the solution throughout the paper. 
\subsection{Surrogate of the Chance constraint}

For the probability $Pr[W(S)>B]$, as the work given in~\cite{xie2019evolutionary}, we consider the one-sided Chebyshev's inequality to construct a usable surrogate of the chance constraint, whose formulation is given below. 

\begin{theorem}[\textit{One-sided Chebyshev's inequality}]
    For any random variable $X$ and $\lambda \geq 0$,
    $Pr[X>E[X]+\lambda]\leq \frac{Var[X]}{Var[X]+\lambda^2}.$
\end{theorem}

For a solution $S$ to the chance-constrained submodular problem, if $Pr[W(S) > B] \le \alpha$, then it is {\it feasible}; otherwise, {\it infeasible}. 
By the one-sided Chebyshev's inequality, we have the following observation directly, which considers the feasibility of a given solution.

\begin{observation}
\label{obs:feasibility}
Given a solution $S$ to the chance-constrained submodular problem,
if $E[W(S)] + \sqrt{\frac{(1-\alpha)Var[W(S)]}{\alpha}}$ $ \leq B$ then the solution $S$ is feasible.
\end{observation}

By the above observation, the {\it surrogate weight} of a solution can be defined as
$\Gamma(S) := E[W(S)] + \kappa_{\alpha}\sqrt{ Var[W(S)]},$
where $\kappa_{\alpha} = \sqrt{\frac{1-\alpha}{\alpha}}$.

\section{Algorithms}
\label{sec:algs}
The first algorithm studied is the greedy algorithm (GA, see Algorithm~\ref{alg:ga}), 
which was analyzed in~\cite{doerr2020optimization} for the chance-constrained submodular problem with all elements having iid weights.
The GA starts with an empty set and picks the element with the largest marginal gain that meets the constraint in each iteration. 
It stops when no more elements can be accepted without violating the constraint.

Considering that the elements may have different weights, the generalized greedy algorithm (GGA, see Algorithm~\ref{alg:gga}) is studied. 
Similar to the mechanism of the GA, the GGA starts with an empty set and stops when no more elements can be added due to the chance constraint. However, the GGA selects the element that satisfies the chance constraint and maximizes the ratio between the additional gain in the objective function $f$ and that in a non-decreasing function $h$. 
As the variances and surrogate weights of the solutions to the problem are non-decreasing, two strategies based on two non-decreasing functions $h$ are respectively studied:
\textbf{Strategy \uppercase\expandafter{\romannumeral1} $h(S) := \sum_{i\in S}\delta_i^2$}, and \textbf{Strategy \uppercase\expandafter{\romannumeral2} $h(S) :=\Gamma(S)$}.
Uncertainties of the solution are only considered in Strategy \uppercase\expandafter{\romannumeral1} and the surrogate weight  based on the one-sided Chebyshev's inequality is studied in Strategy \uppercase\expandafter{\romannumeral2}. 
Furthermore, Lines~\ref{gga:line 10}-\ref{gga:line 11} of Algorithm~\ref{alg:gga} are required when there exists an element with an extremely high objective value, see~\cite{khuller1999budgeted,leskovec2007cost} for more details. 

Additionally, the generalized greedy$+$Max algorithm (GGMA, see Algorithm~\ref{alg:ggma}) is studied. 
The GGMA adopts the same greedy strategies as the GGA, but it uses the feasible item having the largest marginal gain to augment every partial greedy solution. 
More specifically, the augmenting item is selected among the remaining items that still meet the constraint in each iteration.
Until no element can fit into the solution, the GGMA stops and outputs the best-augmented solution. 

Note that the surrogate is applied to the algorithms instead of calculating the probability $Pr[W(S)>B]$.
We only are using the exact calculation for $Pr[W(v)>B]$ when considering a single element at line~\ref{gga:line 10} in the GGA.


\begin{algorithm}[t]
\caption{Greedy Algorithm (GA)}
\label{alg:ga}
\textbf{Input}: \resizebox{0.42\textwidth}{!}{%
Elements set $V$, budget constraint $B$, failure probability $\alpha$
}\\
\textbf{Output}:$S$
\begin{algorithmic}[1]
    \STATE $S \gets \emptyset,V'\gets V$\;\label{line:GA_feasible_options}
    \WHILE{$V'\neq\emptyset$}
    \STATE $v^* \gets \arg\max_{v\in V'}f(S\cup\{v\})-f(S)$\;\label{line:GA_ratio}
    \IF{$\Gamma(S\cup \{v^*\})\leq  B$}
    \STATE $S\gets S\cup \{v^*\}$\;
    \ENDIF
    \STATE $V'\gets V'\setminus\{v^*\}$\;
    \ENDWHILE
\end{algorithmic}
\end{algorithm}

\begin{algorithm}[t]
\caption{Generalized Greedy Algorithm (GGA)}
\label{alg:gga}
\textbf{Input}:\resizebox{0.42\textwidth}{!}{%
Elements set $V$, budget constraint $B$, failure probability $\alpha$
}\\
\textbf{Output}: $S$
\begin{algorithmic}[1]
    \STATE $S \gets \emptyset,V'\gets V$\;\label{line:GGA_feasible_options}
    \WHILE{$V'\neq\emptyset$}
    \STATE $v^* \gets \arg\max_{v\in V'}\frac{f(S\cup\{v\})-f(S)}{h(S\cup \{v\})-h(S)}$\;\label{line:GGA_ratio}
    \IF{$\Gamma(S\cup \{v^*\})\leq  B$}
    \STATE $S\gets S\cup \{v^*\}$\;
    \ENDIF
    \STATE $V'\gets V'\setminus\{v^*\}$\;
    \ENDWHILE
    \STATE $v^*\gets \arg\max_{\{v\in V; Pr[W(v)>B]\le \alpha\}} f(v)$\; \label{gga:line 10}
    \STATE $S\gets\arg\max_{Y\in\{S,\{v^*\}\}}f(Y)$\;\label{gga:line 11}
\end{algorithmic}
\end{algorithm}

\begin{algorithm}[t]
\caption{
\resizebox{0.37\textwidth}{!}{%
Generalized Greedy+Max Algorithm (GGMA)
}}
\label{alg:ggma}
\textbf{Input}: \resizebox{0.42\textwidth}{!}{%
Elements set $V$, budget constraint $B$, failure probability $\alpha$
}\\
\textbf{Output}:$T$
\begin{algorithmic}[1]
    \STATE $T \gets \emptyset,S \gets \emptyset,V'\gets V$\;
    \STATE $V'\gets \left\{v\in V'\setminus S \mid \Gamma(S\cup \{v\}) \leq B \right\}$\;\label{ggma:us}
    \WHILE{$V'\neq\emptyset$}
    \STATE $v' \gets \arg\max_{v\in V'}{f(S\cup\{v\})}$\;\label{line:GGMA_ratio}
    \IF{$f(T)< f(S\cup \{v'\})$}
    \STATE $T\gets S\cup \{v'\}$\;
    \ENDIF
    \STATE $v^* \gets \arg\max_{v\in V'}\frac{f(S\cup\{v\})-f(S)}{h(S\cup \{v\})-h(S)}$\;
    \STATE $S\gets S\cup \{v^*\}$\;
    \STATE Update $V'$ as Line~\ref{ggma:us}\;
    \ENDWHILE
\end{algorithmic}
\end{algorithm}

\section{Performance of the GA}
\label{sec:pga}
According to the previous work~\cite{doerr2020optimization},
the GA is theoretically proven that works well in chance-constrained submodular problems with identical weight and uncertainty.
However, since the uncertainties become variable,
the GA is hard to obtain a high-quality solution, 
which is proved in the below.

Let $S_{cc}$ be the solution obtained by the GA. 
From Theorem~\ref{thm:gap}, we find that the GA performs badly on some linear instances. 
Before the statement, we define such a linear instance $I_1$, 
in which let $V$ have at least $B+1$ elements, $f(S)=|S|$, $\gamma\in(0,1]$, $\alpha\in\left(\frac{3\gamma}{(B-1)^2+3\gamma},\frac{3\gamma}{(B-2)^2+3\gamma}\right)$, $\delta_1=\sqrt{\gamma}$, and $\delta_i=0$ for all $i\geq2$. 
\begin{theorem}
\label{thm:gap}
     There exists a linear instance $I_1$ such that the GA fails to guarantee better than $(1/B)$-approximation.
\end{theorem}
\begin{proof}
    Considering the instance $I_1$, we have $\Gamma(\{1\})\in(B-1,B)$ and the GA on $I_1$ can pick element $1$ in the first iteration, preventing it from continuing. Thus $f(S_{cc})=1$, and the solution is $(1/B)-$ approximation while $Y=\{2,\ldots,B+1\}$, $f(Y)=\Gamma(Y)=B$. The claim is proved.
\end{proof}

The proof reveals that the GA rapidly exhausts the budget (i.e., selecting only one element) due to the significant influence of dispersion in the surrogate. This is the primary factor leading to the suboptimal performance of the GA.

\section{Performance of the GGA}
\label{sec:pgga}


\subsection{Analysis of Using Strategy \uppercase\expandafter{\uppercase\expandafter{\romannumeral1}}}
\label{ssec: gga1}
For Strategy \uppercase\expandafter{\uppercase\expandafter{\romannumeral1}}: $h(S) := \sum_{i\in S}\delta_i^2$, 
we also find that there exists a collection of linear instances of the problem, for which the GGA is hard to obtain a high-quality solution. 
To facilitate the construction of these instances, 
a solution $S$ is encoded as a decision vector $X = x_1 x_2...x_n$ with length $n$, where $x_i=1$ means that the element $i \in V$ is selected into the solution $S$.
Then we define such an instance $I_2$ with a linear function $f$, in which let $V = 1,\ldots,n$, $0 < \alpha <0.5$, and $B= \varepsilon+1$ where $n \ge 2\varepsilon$ and $\varepsilon\geq 1$. 
The function $f$ represented by the decision vector $X$ is given  as:
\begin{equation}
\label{equ:4}
    f(X) = \sum_{i=1}^{\varepsilon}x_i+\varepsilon\sum_{i = \varepsilon+1}^{n}x_i.
\end{equation}
Besides the dispersion of each element in $I_2$ is considered as $\delta_i=\sqrt{\frac{\gamma}{\varepsilon}}$ for $i \in [1,\varepsilon]$, and $\delta_j=\sqrt{\frac{\varepsilon\gamma+\beta}{\varepsilon}}$ for $j \in [\varepsilon+1,n]$ subjected to $0<\gamma$, $0<\beta$ and $\varepsilon\gamma+\beta\leq 3\alpha/(1-\alpha)$,
which indicated by Theorem~\ref{thm:2}.

\begin{theorem}
\label{thm:2}
Given $\varepsilon\geq 1$, there exists a linear instance $I_2$ such that the generalized greedy algorithm GGA applying $h := \sum_{i\in S}\delta^2_i$ fails to guarantee better than $(1/\varepsilon)-$approximation.
\end{theorem}

\subsection{Analysis of Using Strategy \uppercase\expandafter{\uppercase\expandafter{\romannumeral2}}}
Since  $h(S) :=\Gamma(S)$ in Strategy \uppercase\expandafter{\uppercase\expandafter{\romannumeral2}},
we know that $h$ is a non-linear function therefore the surrogate weight of each element is changed as the size of the solution grows.
For the analysis (Theoreom~\ref{the:3}), some useful notations and definitions are introduced first. 
Let $S_{cc}$ be the greedy solution generated by the GGA, 
$v_{i}$ be the $i$-th element added to the solution $S_{cc}$, and $S_i = \{v_1,\ldots v_i\} \subseteq S_{cc}$ ($1 \le i \le |S_{cc}|$) be the set containing the first $i$ elements.
Then we define a set $A_i$ to collect all abandoned elements due to the constraint violation before the GGA adds $v_i$ into $S_{i-1}$. Note that $A_{i-1}\subseteq A_i$. 
Besides, the surrogate weight of the element $v_i$ is denoted by $c_i$, where $c_{i} = \Gamma(S_{i}) - \Gamma(S_{i-1})$. 
Moreover, given any two sets $S, T\subset V$, let $f(S\mid T) := f(S\cup T) - f(T)$.

Let $OPT_d$ be the optimal solution of the deterministic instance of the problem.
Given a partial greedy solution $S_k$ generated by the GGA, 
the relation between $S_k$ and $OPT_d$ is first investigated in Lemma~\ref{lemma:2}. 
Observe that $|OPT_d| = \lfloor B \rfloor$ as the expected weight is exactly one.

\begin{lemma}
    \label{lemma:2}
    Let $\zeta = \kappa_{\alpha}\sum_{j\in OPT_d}\sqrt{\delta_j^2/3}$. Given a partial greedy solution $S_{k}$, if $A_k\cap OPT_d = \emptyset$, then 
    $$f(S_{k+1})-f(S_{k})\geq \frac{c_{k+1}}{\lfloor B \rfloor+\zeta }\cdot (f(OPT_d)-f(S_{k})).$$
\end{lemma}

After that, we can get a relation between $OPT_d$ and $S_{cc}$ by using Lemma~\ref{lemma:2}.

\begin{theorem}
\label{the:3}
The solution obtained by the GGA applying $h := \Gamma(S)$ is a $(1/2-o(1))(1-1/e)-$approximation.
\end{theorem}

 \begin{proof}
    Consider the upper bound of $k$ that is denoted by $k^*$. 
    It has the set $S_{k^*}$ such that the element from $OPT_d$ is first abandoned due to the constraint when the GGA attempts to add it into the set. 
    We denote the abandoned element by $z$ and derive a relation between $S_{k^*}$ and $OPT_d$.
    
    Note that $A_{i}\cap OPT_d = \emptyset$ for $1\leq i\leq k^*$. Following Lemma~\ref{lemma:2},  it gives 
    \begin{equation}
         \label{equ:rskopt}
         f(S_{k^*+1}) -f(S_{k^*}) \geq \frac{c_{k+1}}{\lfloor B\rfloor+\zeta }\cdot(f(OPT_d)-f(S_{k^*})).
    \end{equation}
    As we know that $(1-x)\leq e^{-x}$, then rearranging (\ref{equ:rskopt}) gives 
    \begin{equation}
        \begin{aligned}
            f(OPT_d)&-f(S_{k^*+1}) \\ 
             &\leq \left(1-\frac{c_{k^*+1}}{\lfloor B\rfloor +\zeta }\right)\cdot(f(OPT_d)-f(S_{k^*})) \\
             & \leq e^{-\frac{c_{k^*+1}}{\lfloor B\rfloor +\zeta }} \cdot (f(OPT_d)-f(S_{k^*})).
        \end{aligned}
    \end{equation}
    Recursively,
    \begin{equation}
        \begin{aligned}
            f&(OPT_d) - f(S_{k^*+1}) \\
             &\leq  e^{-\frac{c_{k^*+1}}{\lfloor B\rfloor+\zeta}} \cdot (f(OPT_d)-f(S_{k^*})) \\
             &\leq  e^{-\frac{c_{k^*+1}+c_{k^*}}{\lfloor B\rfloor+\zeta }} \cdot (f(OPT_d)-f(S_{k^*-1})) \\
             &\leq \ldots 
             \leq e^{-\frac{\sum_{i=1}^{k^*+1} c_i}{\lfloor B\rfloor+\zeta }} \cdot f(OPT_d) \\
             &=  e^{-\frac{\Gamma(S_{k^*+1})}{\lfloor B\rfloor+\zeta }} \cdot f(OPT_d),
        \end{aligned}
    \end{equation}
    Consequently, we get the relation between $S_{k^*}$ and $OPT_d$ as 
    \begin{equation}
    \label{equ:rso}
 f(S_{k^*+1}) \geq \left(1-e^{-\Gamma(S_{k^*+1})/(\lfloor B\rfloor+\zeta)}\right)\cdot f(OPT_d).
    \end{equation}
    
    Then we investigate the approximation by using the relation and the abandoned element $z$.
    By Observation~\ref{obs:feasibility} and definitions, it observes that $\Gamma(S_{k^*}\cup \{z\}) = \Gamma(S_{k^*}) + c' > \lfloor B\rfloor$,
    where $c' = \Gamma(z\mid S_{k^*})$. Putting it with (\ref{equ:rso}) together gives 
    \begin{equation}
        \begin{aligned}
            &f(S_{k^*+1}) \geq \left(1-e^{-\frac{\Gamma(S_{k^*+1})}{\Gamma(S_{k^*})+c'+\zeta}}\right)\cdot f(OPT_d) \\
            &= \left(1-e^{-1}exp\left({\frac{\zeta+c'-c_{k^*+1}}{\Gamma(S_{k^*})+c'+\zeta}}\right)\right)\cdot f(OPT_d).
        \end{aligned}
    \end{equation}
    As $S_{k^*+1}$ at least include one element, the expression of $exp\left(\cdot\right)$ is $(1+o(1))$.
    Moreover, let $v^*\in V\setminus S_{k^*}$ be the element that has the largest function value. Observe that $f(v^*) \geq f(v_{k^*+1})$ and $f(S_{cc})>f(S_{k^*})$.
    It gets $f(S_{cc}) + f(v^*) \geq f(S_{k^*}) + f(v^*) \geq f(S_{k^*+1})$.
    Putting them together gets $f(S_{cc}) + f(v^*)\geq (1-o(1))(1-1/e)\cdot f(OPT_d)$, and therefore  $\max_{Y\in\{S_{cc},\{v^*\}\}}f(Y) \geq (1/2-o(1))(1-1/e)\cdot f(OPT_d).$
 \end{proof}

\section{Performance of the GGMA}
\label{sec:pggma}


In this section, we analyze the approximation behavior of the GGMA. 
The performance of the algorithm applying two different strategies is investigated separately.
\subsection{Analysis of Using Strategy \uppercase\expandafter{\uppercase\expandafter{\romannumeral1}}}

Theorem~\ref{thm: pgga2} implies that using Strategy \uppercase\expandafter{\uppercase\expandafter{\romannumeral1}}: 
$h(S):=\sum_{i\in S}\delta_i^2$, 
the GGMA also performs badly in the instances $I_2$, 
which was presented in the Section~\ref{ssec: gga1}.
\begin{theorem}
\label{thm: pgga2}
    Given $\varepsilon\geq 1$, there exists an instance $I_2$ such that the GGMA applying $h := \sum_{i\in S}\delta^2_i$ fails to guarantee better than $(2/\varepsilon - 1/\varepsilon^2)-$approximation.
\end{theorem}


\subsection{Analysis of Using Strategy \uppercase\expandafter{\uppercase\expandafter{\romannumeral2}}}
For Strategy \uppercase\expandafter{\uppercase\expandafter{\romannumeral2}}: $h(S):=\Gamma(S)$, 
let $S_{cc}$ be the greedy solution constructed by the greedy strategy for the instance in the chance-constrained setting, 
and $OPT_d$ be the optimal solution for the corresponding deterministic instance.
In the $i$-th generation, the element $v_i$ is selected by the algorithm, and its surrogate weight is denoted by $c_i$. 
Besides, the partial solution containing the first $i$ item is denoted by $S_i\subseteq S_{cc}$.  
Then some useful greedy performance functions are defined to track the performance of the algorithm for the chance-constrained setting.

For a fixed $x \in [0, B]$, let $i$ be the smallest greedy index so that $\Gamma(S_i)> x$. 
Then to track the performance of the greedy strategy, a continuous and monotone piecewise-linear function $g(x)$ is defined as
$g(x) = f(S_{i-1})+(x-\Gamma(S_{i-1}))\frac{f(v_i|S_{i-1})}{c_i}$,
and $g(0):= 0$. 
Additionally, $g'$ denotes the right derivative for $g$ on the interval $[0, \Gamma(S_{cc}))$. 
Observe that $g'$ is always non-negative as the objective value of the greedy solution does not decrease after including a new item. 
Besides, to track the performance of the greedy$+$Max when the greedy solution collects a set of cost $x$, we define the function $g_+(x) = g(x) + f(v\mid S_{i-1})$, where 
$v = \argmax_{j\in V\setminus S_{i-1}:\Gamma(\{j\}\cup S_{i-1})\leq B} f(j\mid S_{i-1})$.

After that, we consider the lower bound of the function $g_+$ in the specific interval.
Following the definition of the fixed greedy index $i$, $z_{max}$ denotes the element that has the largest dispersion in $OPT_d\setminus S_{i-1}$. 
Note that $z_{max}$ is the first abounded element from $OPT_d$ due to the chance constraint.
Denote by $S_{k^*}$ the partial solution.
If $S_{k^*}$ is obtained then $z_{max}$ is removed.
That implies $z_{max}$ can be selected by the algorithm as the augmenting item for the set $S_i$ where $0\leq i\leq  k^*-1$.
Therefore for $x \in [0, \Gamma(S_{k^*-1})]$, we define the greedy$+$Max performance lower bound as $g_1(x) = g(x) + f(z_{max}\mid S_{i-1})$, so that $g_1\leq g_+$ for $x \in [0, \Gamma(S_{k^*-1})]$.

Now we investigate the relation between $OPT_d$, $g_1(x)$ and $g'(x)$ while $x \in [0, \Gamma(S_{k^*-1})]$ in Lemma~\ref{lemma:3} (proof in Appendix \ref{sect_app_proof_lemma3}).

\begin{lemma}
\label{lemma:3}
For any $x \in [0, \Gamma(S_{k^*-1})]$, let $i$ be the smallest greedy index so that $\Gamma(S_i)> x$.
It holds that $$f(OPT_d)  \leq g_1(x) + g'(x)\sum_{j\in OPT_d \setminus z_{max}} \Gamma(j|S_{i-1}).$$
\end{lemma}

Then we focus on the point $x = \Gamma(S_{k^*-1})$ and analyze the approximation behavior of the GGMA (Thereom~\ref{the:5}) via Lemma~\ref{lemma:3}. 
\begin{theorem}
\label{the:5}
    The solution obtained by the GGMA applying $h := \Gamma(S)$ is a $(1/2-o(1))-$approximation.
\end{theorem}
\begin{proof}
    Following Lemma~\ref{lemma:3} and applying $x = \Gamma(S_{k^*-1})$, it gives
    \begin{equation}
        \begin{aligned}
            f(OPT_d) &\leq g_1(\Gamma(S_{k^*-1}))\\
            &+ g'(\Gamma(S_{k^*-1}))\sum_{j\in OPT_d \setminus Z_{max}} \Gamma(j|S_{k^*-1}).
        \end{aligned}
    \end{equation}
    Here the upper bound of the last term is given below. Recall that $|OPT_d| = \lfloor B\rfloor$. Let $\psi := \sqrt{Var[W(S_{k^*-1})]}$. It holds that 
    \begin{equation}
        \begin{aligned}
             &\sum_{j\in OPT_d \setminus z_{max}} \Gamma(j|S_{k^*-1})\\ 
            &=\sum_{j\in OPT_d} \Gamma(j|S_{k^*-1}) - \Gamma(z_{max}|S_{k^*-1}) \\
            & \leq \sum_{j\in OPT_d} \Gamma(j|S_{k^*-1}) - \Gamma(z_{max}|S_{k^*})  \\ 
            &=\lfloor B\rfloor + \eta - c'_{max},
        \end{aligned}
    \end{equation}
    where $\eta = \kappa_{\alpha}\sum_{j\in OPT_d}  \left(\sqrt{Var[W(S_{k^*-1}\cup \{j\})]}- \psi\right)$ and  $c'_{max} = \Gamma(z_{max}|S_{k^*})$.
    Consequently, we have
    \begin{equation}
 \label{equ:10}
 \begin{aligned}
 f(OPT_d) \leq g_1&(\Gamma(S_{k^*-1})) \\
 &+ g'(\Gamma(S_{k^*-1}))( \lfloor B\rfloor + \eta -c'_{max}).
 \end{aligned}
    \end{equation}
    
    After that, we consider the value of $g_1(\Gamma(S_{k^*-1}))$ in (\ref{equ:10}). Recall that the surrogate weight of $v_{k^*}$ is $c_{k^*}$. 
    Then let $c^* := 1+\kappa_{\alpha}\sqrt{\delta_{k^*}^2/3}$  and 
    $\phi := \frac{c_{k^*}\cdot(\Gamma(S_{k^*-1})+c^*)}{\Gamma(S_{k^*})+\eta}$.
    Two possible cases for it are listed as follows.

    \textbf{Case 1}. 
    $g_1(\Gamma(S_{k^*-1}))\geq \frac{\phi}{c^*+\phi}\cdot f( OPT_d)$. 
    Since $g_1(x)\leq g_+(x)$ for $x\in [0, \Gamma(S_{k^*-1})]$, it directly holds $g_+(\Gamma(S_{k^*-1}))\geq \frac{\phi}{c^*+\phi}\cdot f(OPT_d)$. 
    
    \textbf{Case 2}. $g_1(\Gamma(S_{k^*-1}))< \frac{\phi}{c^*+\phi}\cdot f( OPT_d)$. For this case, we can prove that $g(\Gamma(S_{k^*}))\geq \frac{\phi}{c^*+\phi}\cdot f(OPT_d)$ as following. 
    First of all, rearranging (\ref{equ:10}) gets
    \begin{equation}
    \label{equ:gphi}
 \begin{aligned}
 g'(\Gamma(S_{k^*-1})) &\geq \frac{f( OPT_d)-g_1(\Gamma(S_{k^*-1}))}{\lfloor B\rfloor+ \eta - c'_{max} }\\
 &\geq \frac{c^*\cdot f( OPT_d)}{(c^*+\phi)(\lfloor B\rfloor+ \eta - c'_{max})}.
 \end{aligned}
    \end{equation}
    Besides, let $g'_{min} := \arg\min_{i\in[1,k^*]} g'(\Gamma(S_{i-1}))$. 
    Recall $g'$ is non-negative and $g(0)=0$, thus it holds that 
    \begin{equation}
        g(x)\geq x \cdot g'_{min},
    \end{equation} 
    for any $x\in [0,\Gamma(S_{k^*-1})]$.
    Now we show that $g'_{min} \geq \frac{f(v_{k^*}\mid S_{k^*-1})}{c^*}$.
    Considering the greedy strategy of the GGMA, it holds that  
    $ g'(\Gamma(S_{i-1})) = \frac{f(v_i|S_{i-1})}{c_{i}}\geq \frac{f(v_{k^*}|S_{i-1})}{\Gamma({v_{k^*}|S_{i-1}})}$
    in the $i$-th generation for $1\leq i \leq k^*$.
    Observe $\Gamma({v_{k^*}|S_{i-1}}) \leq c^*$ and $f(v_{k^*}|S_{i-1})\geq f(v_{k^*}|S_{k^*-1})$ for $ i \leq k^*$ as $S_{i-1}\subseteq S_{{k^*-1}}$ (recall (\ref{equ:1})). Therefore we have $\frac{f(v_{k^*}|S_{i-1})}{\Gamma({v_{k^*}|S_{i-1}})}\geq \frac{f(v_{k^*}|S_{k^*-1})}{c^*}.$ 
    Putting them together yields $g'_{min} \geq \frac{f(v_{k^*}\mid S_{k^*-1})}{c^*}$.
    For any $x \in [0,\Gamma(S_{k^*-1})]$, therefore it holds that 
    \begin{equation}
    \label{equ:11}
 g(x) \geq x\cdot \frac{f(v_{k^*}|S_{k^*-1})}{c^*}.
    \end{equation}
    Then applying $x = \Gamma(S_{k^*-1})$ to (\ref{equ:11}), it gets 
    \begin{equation}
        \begin{aligned}
            g(\Gamma(S_{k^*-1})) &\geq \Gamma(S_{k^*-1})\cdot \frac{f(v_{k^*}\mid S_{k^*-1})}{c^*}
        \end{aligned}
    \end{equation}
    Besides, since $g'(\Gamma(S_{k^*-1})) = \frac{f(v_{k^*}\mid S_{k^*-1})}{c_{k^*}}$, rearranging (\ref{equ:gphi}) gets
    $\frac{f(v_{k^*}\mid S_{k^*-1})}{c^*}\geq \frac{c_{k^*}\cdot f( OPT_d) }{(c^*+\phi)(\lfloor B\rfloor+ \eta - c'_{max})}.$
    Putting them together, we have 
    $g(\Gamma(S_{k^*-1}))\geq \frac{f( OPT_d)\cdot\Gamma(S_{k^*-1})\cdot c_{k^*}}{(c^*+\phi)\left(\lfloor B\rfloor+ \eta - c'_{max}\right)}.$

    Now we can derive a lower bound for the objective value of the set $S_{k^*}$. Recall that $v_{k^*}$ is the next added element for the solution $S_{k^*-1}$.  Thus $f(S_{k^*})$ is at least 
    \begin{equation}
        \begin{aligned}
                 g&(\Gamma(S_{k^*}))= g(\Gamma(S_{k^*-1})) + c_{k^*} g'(\Gamma(S_{k^*-1}))\\
         & \geq \frac{ c_{k^*}\cdot(\Gamma(S_{k^*-1}) + c^*)}{(c^*+\phi)(\lfloor B\rfloor+ \eta - c'_{max} )}\cdot f( OPT_d).
        \end{aligned}
    \end{equation}
    Furthermore, by Observation~\ref{obs:feasibility} it yields that 
    $\Gamma(S_{k^*-1}\cup\{v_{k^*}, z_{max}\}) = \Gamma(S_{k^*}) + c'_{max} \geq \lfloor B\rfloor.$ 
    Put them together gets 
    \begin{equation}
        \begin{aligned}
             g(\Gamma(S_{k^*}))&\geq \frac{ c_{k^*}\cdot(\Gamma(S_{k^*-1}) + c^*)}{(c^*+\phi)(\lfloor B\rfloor+ \eta - c'_{max} )}\cdot f(OPT_d)\\ 
            &\geq\frac{c_{k^*}\cdot(\Gamma(S_{k^*-1})+c^*   )}{(c^*+\phi)(\Gamma(S_{k^*})  + \eta)}\cdot f( OPT_d) \\ 
            &= \frac{\phi}{c^*+\phi}\cdot f( OPT_d).
        \end{aligned}
    \end{equation}
    Therefore the GGMA achieves a $\frac{\phi}{c^*+\phi}$-approximation.
    Since at least one element is included in the greedy solution, the expression of $\frac{\phi}{c^*+\phi}$ is $(1/2-o(1))$. 
    Additionally, as the objective value of the augmented output solution $T$ is no worse than $g_+(\Gamma(S_{k^*}))$, it holds 
    $f(T)\geq (1/2-o(1))\cdot f(OPT_d)$.
\end{proof}

\section{Experiments}
\label{sec:exp}

\begin{figure*}[t]
    \centering
    \includegraphics[width = \textwidth]{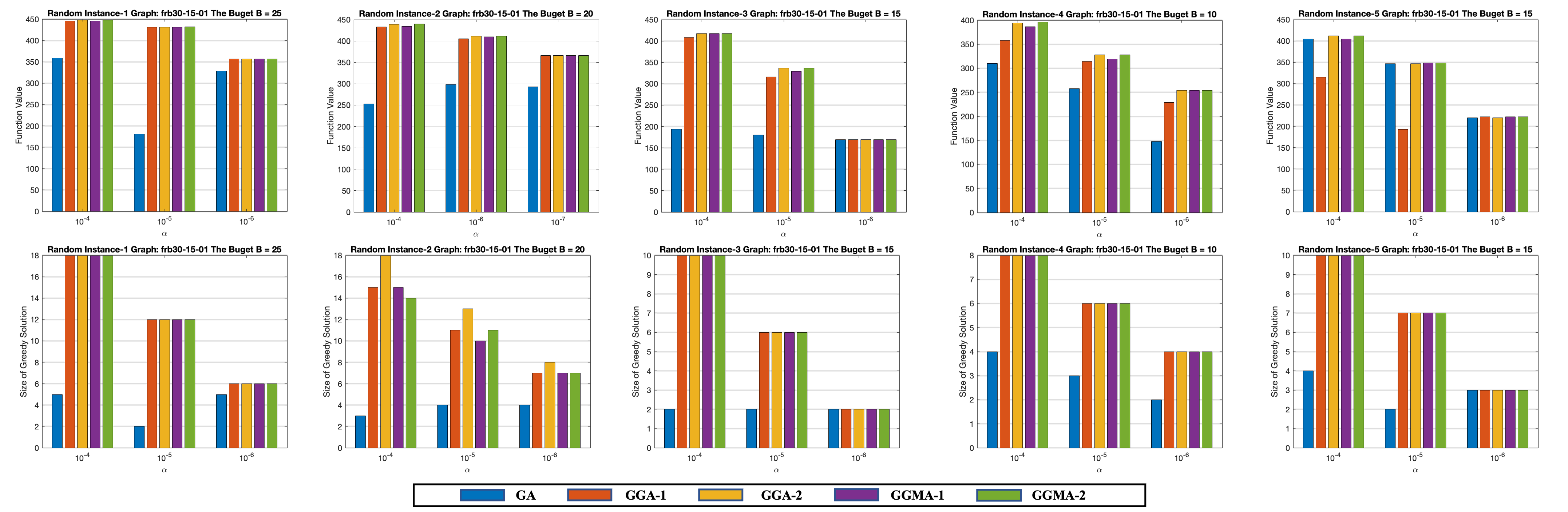}
    \vspace*{-5mm}
    \caption{$N(V')$ (top) and $|V'|$ (bottom) for the graph frb30-15-01 with different budgets when $\delta$ is randomly sampled from the uniform distribution.}
    \vspace*{-2mm}
    \label{fig:1}
\end{figure*}

\begin{figure*}[t]
    \centering
    \includegraphics[width = 0.99\textwidth]{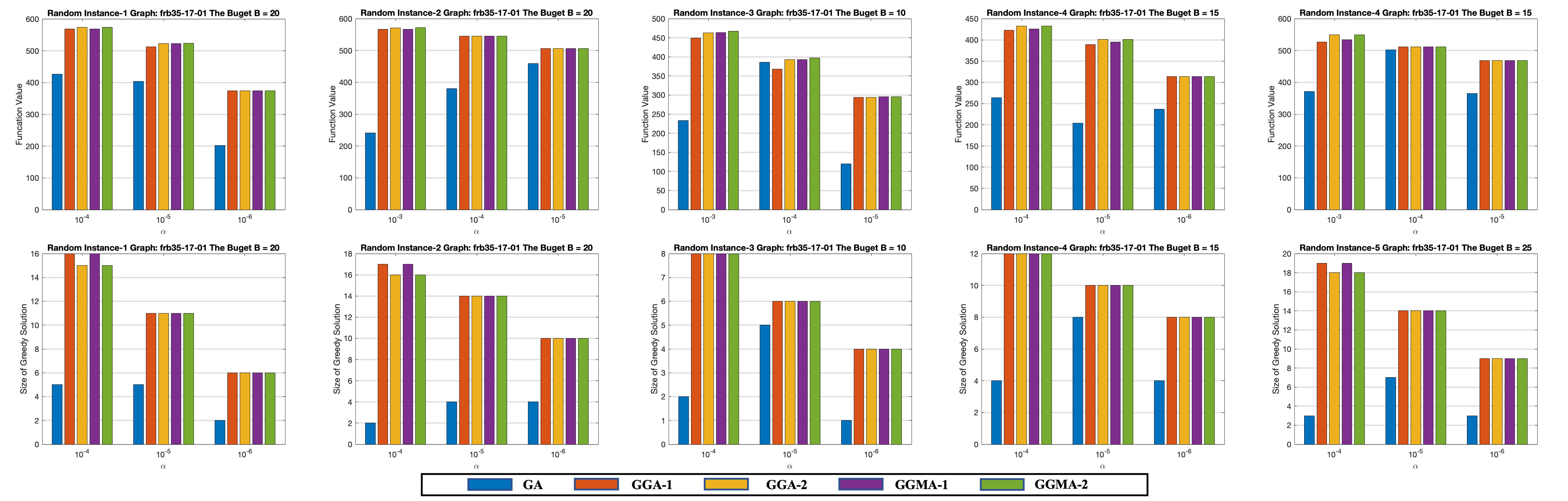}
    \vspace*{-2mm}
    \caption{$N(V')$ (top) and $|V'|$ (bottom) for the graph frb35-17-01 with different budgets when $\delta$ is randomly sampled from the uniform distribution.}
    \vspace*{-2mm}
    \label{fig:2}
\end{figure*}

\begin{figure*}[t]
    \centering
    \includegraphics[width = 0.9\textwidth]{  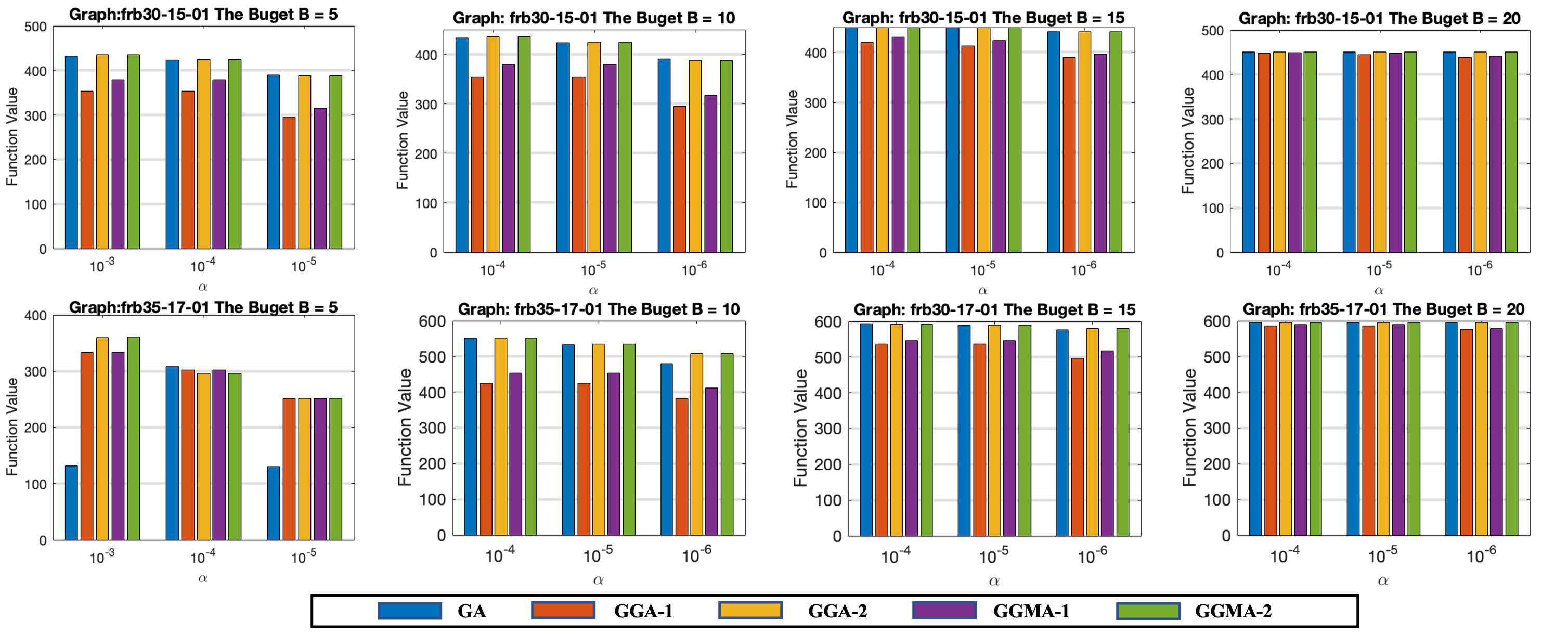}
    \vspace*{-2mm}
    \caption{$N(V')$ for the graphs frb30-15-01 (top) and frb35-17-01 (bottom) with different budgets when $\delta$ is based on the degree.}
    \label{fig:3}
    \vspace*{-2mm}
\end{figure*}

\begin{figure*}[t]
    \centering
    \includegraphics[width =\textwidth]{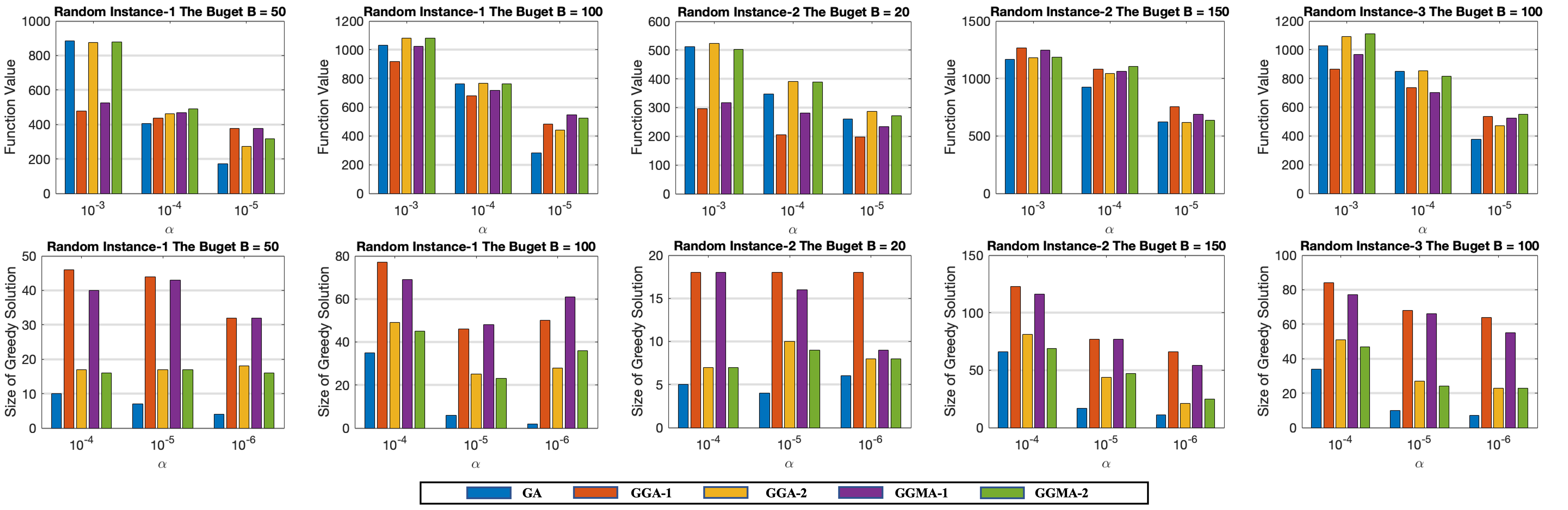}
    \vspace*{-5mm}
    \caption{$E[I(X)]$ (top) and $|X|$ (bottom) for IMP with different budgets when $\delta$ is randomly sampled from the uniform distribution.}
    \label{fig:4}
    \vspace*{-2mm}
\end{figure*}

In this section, we regard the GA as the baseline algorithm and evaluate the experimental performance of other algorithms (namely the GGA and the GGMA) on two significant submodular optimization problems such as the maximum coverage problem (MCP) and the influence maximum problem (IMP) with chance constraint. 
Following the specific setting described in Section~\ref{sec:pl}, the expectation of each element's weight is set as $1$, and the dispersion value of each item is different.

\subsection{The Maximum Coverage Problem}
The first submodular problem is the maximum coverage problem~\cite{feige1998threshold,khuller1999budgeted}.
We consider a chance-constrained version of the MCP based on the graph.
Given an undirected graph $G = (V,E)$, 
we denote the degree of the node $v_i$ by $D(v_i)$,
and the number of all nodes of $V'\subseteq V$ and their neighbors in $G$ by the objective function $N(V')$. 
The MCP aims to find a subset $V'$ so that $N(V')$ is maximized under the constraint. 
Moreover, given a linear cost function $c: V\to \mathbb{R}^{+}$, 
the problem under chance constraint is formulated as 
\begin{equation}
    \argmax_{V'\subseteq V}~N(V')~s.t.~Pr[c(V')>B]\leq \alpha.
\end{equation}

The graph used in the instances are frb30-15-01 (450 nodes and 17827 edges) and frb35-17-01 (595 nodes and 27 856 edges)~\cite{nr}.
In terms of settings, for each node $v_i$, the cost $a_i$ is 1 but the value of the dispersion $\delta_i$ is set by two different methods. 
The first method is that $\delta_i$ is independently uniformly at random sampled from $[0,1]$.
To analyze our experiments more rigorously, we independently randomly sample the value of dispersion five times for each graph.
Moreover, we also consider $\delta_i$ is associated with the degree of the node, which can be expressed as $\delta_i = D(v_i)/\sum_{v\in V}D(v)$.
Furthermore, we investigate all combinations of $\alpha \in \{10^{-4},10^{-5}, 10^{-6}\}$ and $B \in \{10, 15, 20, 25\}$ for the experimental investigation of the algorithms. 
The performance of the algorithms is measured in terms of the function value $N(V')$.

The experimental results are shown in figures~\ref{fig:1},~\ref{fig:2} and~\ref{fig:3},
which indicate that
for the instances with the same budget, both the function value and the number of nodes of the solution by the algorithms GGA and GGMA decline 
as $\alpha$ increasing. 
They also show
that the performance of the GA is worse than GGA and GGMA using strategy \uppercase\expandafter{\romannumeral2} among most instances.
It also can be found that GA collects fewer items before reaching the budget,
which matches our theoretical analysis. 
Besides, for the same strategy applied to the different algorithms, 
the GGMA slightly outperforms the GGA while applying strategy \uppercase\expandafter{\romannumeral1} 
to the instances with a lower budget.
Additionally, the performance of the GGMA using strategy \uppercase\expandafter{\romannumeral2} is comparable to the GGA.

In terms of strategies,
we observe from figures that using strategy \uppercase\expandafter{\romannumeral2} can improve the performance of all algorithms
by strategy \uppercase\expandafter{\romannumeral1}.
More precisely, the algorithms with strategy \uppercase\expandafter{\romannumeral1} can output a solution that includes more nodes but is with a lower function value, among most instances.
It is noticeable that the GGMA with strategy \uppercase\expandafter{\romannumeral2} can obtain high-quality solutions for these instances.

\subsection{The Influence Maximization Problem}
We now study the influence maximization problem~\cite{zhang2016submodular,qian2017subset,leskovec2007cost}.
The IMP aims to identify the set of users, 
who are the most influential in a large-scale social network.

The goal of the IMP is to maximize the spread of influence over a given social network, 
i.e., a graph of interactions within a group of users~\cite{kempe2003maximizing}. 
This section presents the experimental analysis of some chance-constrained IMP instances.

Let a directed graph be $G = (V,E)$ to represent a social network, 
in which 
each node $v_i\in V$ corresponds to a user, and
the probability $p_{i,j}$ of the edge in $E$ represents the strength of the influence between a pair of users $v_i$ and $v_j$.
The IMP aims to find a subset $X \subseteq V$ such that the expected number of nodes $E[I(X)]$ (the objective function of the problem) activated by propagating from $X$ is maximized subject to the constraints.
Given a linear cost function $c: V\to \mathbb{R}^{+}$ and a budget $B$, 
the chance constraint version of the IMP is formulated as
\begin{equation}
    \argmax_{X\subseteq V}~E[I(X)]~s.t.~ Pr[c(X)> B]\leq \alpha.
\end{equation}

The dataset \textit{Social circles: Facebook} consists of friends lists collected from a social networking service, which includes 4,039 nodes and 176,468 edges ~\cite{leskovec2012learning}.
We transform the instances in this dataset to chance-constrained IMP instances.
For each node $v_i$, its expected cost is set as 1, and the dispersion $\delta_i$ is independently and uniformly sampled from $[0,1]$.
The algorithms are evaluated for all pairs of budgets $B\in \{20,50,100,150\}$ and tolerate probabilities $\alpha \in \{10^{-3},10^{-4},10^{-5}\}$.
We use the function value $E[I(X)]$ to evaluate 
the performance of algorithms, and independently sample the value of dispersion three times to analyze our experiments more rigorously.

Figure~\ref{fig:4} clearly shows that
the GA collects fewer nodes than other algorithms and performs unwell when $\alpha$ is small.
For the same budget instances, the function value obtained sharply decreases with the increasing value of $\alpha$. 
This phenomenon is common among those algorithms.
Moreover, we observe that the GGMA includes fewer nodes than GGA but obtains higher function values in most instances.

In terms of strategies, 
the results demonstrate that 
the algorithms with strategy \uppercase\expandafter{\romannumeral1} is significantly worse in some instances than them with strategy \uppercase\expandafter{\romannumeral2}, 
even worse than the GA although it collects more elements. 
On the other hand, the algorithms applying strategy \uppercase\expandafter{\romannumeral2} can fix it in most instances, which coincides with our theoretical analysis. 
In addition,
the GGMA beats the GGA in terms of the quality of the output solution in most instances.

\section{Conclusion}
\label{sec:con}
The paper studied a chance-constrained submodular optimization problem with variable uncertainties
and investigated the performance of the GA, the GGA, and the GGMA on it. 
In the setting, the weights of elements are sampled from distributions with the same expectation but varied dispersion values. 
We found that the GA does not perform well even in some linear instances. 
Besides the GGA and the GGMA respectively can achieve guarantee a $(1/2-o(1))(1-1/e)-$approximation 
and a $(1/2-o(1))-$approximation of the optimal solution for a deterministic setting. 
Additionally, the experimental results showed that the GGMA using the surrogate weight based on the one-sided Chebyshev's inequality beats 
other algorithms in some instances of the MCP and the IMP which are typical submodular problems.

The future work is   to broaden the exploration of chance-constrained submodular problems with variable weights and uncertainties, and potentially different distributions. These subsequent studies will be both challenging and engaging, with the aim of yielding more meaningful insights to enhance our comprehension of the problem.

\ack This work has been supported by the Australian Research Council (ARC) through grant FT200100536, 
the Hunan Provincial Natural Science Foundation of China through grant 2021JJ40791,
and the Open Project of Xiangjiang Laboratory (No.22XJ03005).

\bibliography{mainRef}

\newpage
\appendix

\section{Proofs}

\subsection{Proof of Lemma~\ref{lemma:2}}
\label{sect_app_proof_lemma2}
\begin{proof}
    By the monotonicity and submodularity, it has
    \begin{equation}
    \label{equ:suq}
 \begin{aligned}
 f(OPT_d)&\leq f(OPT_d\cup S_k) \\
 &\leq f(S_k) + f(OPT_d\setminus S_k|S_k)\\
 &\leq f(S_k) +\sum_{{j}\in OPT_d\setminus S_k} \Gamma(j|S_k)\cdot\frac{f(j\mid S_k)}{\Gamma(j|S_k)},
 \end{aligned}
    \end{equation}
    where $\Gamma(j|S_k) = \Gamma(S_k\cup \{j\}) - \Gamma(S_k)$. Then according to the greedy strategy, for any element $j\in OPT_d\setminus S_k$, it gets  
    $\frac{f({v_{k+1}}\mid S_k)}{c_{k+1}} \geq \frac{f(j\mid S_k)}{\Gamma(j|S_k)}$,
    since $OPT_d\cap A_k = \emptyset$. Consequently putting it with (\ref{equ:suq}) together gives that
    \begin{equation}
    \label{equ:ggap}
    \begin{aligned}
 f(OPT_d) \leq f&(S_k)\\
 &+ \frac{f(v_{k+1}|S_k)}{c_{k+1}} \sum_{j\in OPT_d\setminus S_k}&\Gamma(j|S_k) .
    \end{aligned}
    \end{equation}
    There is an upper bound for $\sum_{j\in OPT_d\setminus S_k}\Gamma(j|S_k)$. 
    As it has $\Gamma(j|S_k) \leq 1+\kappa_{\alpha} \sqrt{\delta_j^2/3}$ for any $j\in OPT_d\setminus S_k$, it holds that 
    $\sum_{j\in OPT_d\setminus S_k}\Gamma(j|S_k) \leq \lfloor B\rfloor + \zeta$,
    where $ \zeta = \kappa_{\alpha} \sum_{j\in OPT_d}\sqrt{\delta_j^2/3}$. 
    Substituting it into (\ref{equ:ggap}) completes the proof.
\end{proof}

\subsection{Proof of Lemma~\ref{lemma:3}}
\label{sect_app_proof_lemma3}
\begin{proof}
    It suffices to show the statement only for the points $x = \Gamma(S_{i-1})$ where $1\leq i \leq k^*$. 
    Let $S'_{i-1}:=S_{i-1}\cup z_{max}$. 
    Obverse that $g_1(x) = g(\Gamma(S_{i-1})) + f(z_{max}\mid S_{i-1})= f(S_{i-1}\cup z_{max}) = f(S'_{i-1}).$
    By monotonicity and submodularity, it gives that
    \begin{equation}
    \label{equ:7}
 \begin{aligned}
 f(OPT_d) & \leq f(S_{i-1}\cup OPT_d) \\
 &= f(S'_{i-1}) + f(OPT_d\setminus (S'_{i-1})\mid S'_{i-1}\}) \\
 & \leq g_1(x) + \sum_{j\in OPT_d\setminus S'_{i-1}} f(j\mid S'_{i-1}) 
 \end{aligned}
    \end{equation}
    Since $S_{i-1}\subset S'_{i-1}$, it holds
    $ f(j\mid S'_{i-1}) \leq  f(j\mid S_{i-1})$ for any element $j\in OPT_d\setminus S'_{i-1}$ by submodularity (recall the inequality (\ref{equ:1})).
    Furthermore, recall that $z_{max} $ has the largest dispersion  in $OPT_d\setminus S_{i-1}$ and it still can be added into $S_i$, thus any elements from $OPT_d\setminus S'_{i-1}$ also can be added into $S_i$ without violating the constraint. According to the greedy strategy, it gives that 
    $\frac{f(j\mid S_{i-1})}{\Gamma(j|S_{i-1}) } \leq \frac{f(v_i\mid S_{i-1})}{c_i} = g'(x),$
    for any element $j\in OPT_d\setminus S'_{i-1}$.
    Consequently, putting them with (\ref{equ:7}), it gets that 
    \begin{equation}
    \label{equ:8}
 \begin{aligned}
 &f(OPT_d) \\
 & \leq g_1(x) + \sum_{j\in OPT_d\setminus S'_{i-1}} f(j\mid S'_{i-1}) \\
 &\leq g_1(x) + \sum_{j\in OPT_d\setminus S'_{i-1}} \Gamma(j|S_{i-1}) \frac{f(j\mid S_{i-1})}{\Gamma(j|S_{i-1})} \\
 &\leq g_1(x) + \sum_{j\in OPT_d\setminus S'_{i-1}} \Gamma(j|S_{i-1}) \cdot g'(x) \\
 &\leq g_1(x) + g'(x) \sum_{j\in OPT_d\setminus z_{max}} \Gamma(j|S_{i-1}).\\
 \end{aligned}
    \end{equation}
The proof is completed.
\end{proof}

\subsection{Proof of Theorem~\ref{thm:2}}
\label{sect_app_proof_thm:2}
\begin{proof}

We consider the instance $I_2$. Given a $\varepsilon-$size solution $X$, it can be found that the largest surrogate weight of $X$ is $\varepsilon+\sqrt{\frac{1-\alpha}{3\alpha}(\varepsilon\gamma+\beta)}\leq \varepsilon+1$. Thus every $\varepsilon-$size solution is feasible for $I_2$. Observe that any $(\varepsilon+1)-$size solution cannot be accepted.

Then we analyze the solution returned by the GGA for $I_2$. 
For the first $\varepsilon$ elements and the remaining, the marginal ratios between $f$ and $h$ are $\frac{\varepsilon}{\gamma}$ and $\frac{\varepsilon^2}{\varepsilon\gamma+\beta}$, respectively.
By the greedy strategy of the GGA, the algorithm always selects the items from the first $\varepsilon$ of them since $\frac{\varepsilon}{\gamma}> \frac{\varepsilon^2}{\varepsilon\gamma+\beta}$ until no more element can be added. 
Consequently, the algorithm returns the solution $X$ that contains all the first $\varepsilon$ elements, and $f(X)=\varepsilon$. Moreover, considering line~\ref{gga:line 10} in the GGA, the algorithm gets $f(v^*) = \varepsilon.$ Finally, GGA outputs the solution $S_{cc}$ with the value $f(S_{cc}) = \varepsilon$.

Besides, we define a feasible solution $Y\subseteq \{\varepsilon+1,\ldots, n\}$ that has $\varepsilon$ elements.
It has $f(Y) = \varepsilon^2$ by Equation~(\ref{equ:4}). Denote the optimal solution of $I_2$ for the deterministic setting by $OPT_d$. Obviously, $f(OPT_{d}) \geq f(Y)$ and $f(S_{cc})/f(OPT_{d}) \leq f(S_{cc})/f(Y)=1/\varepsilon$. The claim is proved.
\end{proof}

\subsection{Proof of Theorem~\ref{thm: pgga2}}
\label{sect_app_proof_pgga2}
\begin{proof}
    Considering the given linear instance $I_2$, recall that any $\varepsilon-$size solution is feasible. 
    Because the marginal ratio between the additional gain in $f$ and $h$ are $\frac{\varepsilon}{\gamma}$ and $\frac{\varepsilon^2}{\varepsilon\gamma+\beta}$ for the first $\varepsilon$ elements and rest of them respectively, the GGMA constructs a $(\varepsilon-1)-$size partial solution $X\subseteq \{1,\ldots,\varepsilon\}$ and selects one augmenting item $j\in \{\varepsilon+1,\ldots n\}$ with objective value $\varepsilon$. 
    Thus the output solution is $S_{cc} = X\cup \{j\}  $ with $f(S_{cc}) = 2\varepsilon - 1$. 
    Besides it has $f(OPT_{d})\geq \varepsilon^2$.
    Observe that $f(S_{cc})/f(OPT_{d}) \leq 2/\varepsilon - 1/\varepsilon^2$. The proof is completed.
\end{proof}

\end{document}